\documentclass[11pt,reqno]{amsart}
\usepackage{hyperref}
\hypersetup{
	colorlinks=true,
	linkcolor=blue,
	filecolor=blue,
	urlcolor=red,
	citecolor=green,}

\usepackage{CJK}
\usepackage{amssymb, amsthm}
\makeatletter

\newcommand{\Rmnum}[1]{\expandafter\@slowromancap\romannumeral #1@}
\makeatother

\newtheorem{thm}{Theorem}[section]
\newtheorem{cor}[thm]{Corollary}
\newtheorem{lem}[thm]{Lemma}

\newtheorem{rem}{Remark}

\theoremstyle{definition}
\newtheorem{defn}[thm]{Definition}
\theoremstyle{remark}

\numberwithin{equation}{section}

\usepackage{bm}

\newcommand{\thmref}[1]{Theorem~\ref{#1}}

\newcommand{\lemref}[1]{Lemma~\ref{#1}}

\newcommand{\corref}[1]{Corollory~\ref{#1}}

\newcommand{\R}{\mathbb{R}}

\newcommand{\Lie}{\mathcal{L}}

\newcommand{\ric}{{\rm Ric}}

\newcommand{\di}{{\rm div}}

\newcommand{\la}{\bm{\langle}}
\newcommand{\ra}{\bm{\rangle}}

\newcommand{\vol}{{\rm{vol}}}

\begin{document}
	
	\title[Gradient estimate for $\Delta_pv +av^{q}=0$]{Gradient estimate for solutions of the equation $\Delta_pv +av^{q}=0$ on a complete Riemannian manifold}
	

	\author{Jie He}
	\address{College of Mathematics and Physics, Beijing University of Chemical Technology,  Chaoyang District, Beijing 100029, China}
	\email{hejie@amss.ac.cn}
	
	\author{Youde Wang}
	\address{1. School of Mathematics and Information Sciences, Guangzhou University;
		2. Hua Loo-Keng Key Laboratory of Mathematics, Institute of Mathematics, Academy of Mathematics and Systems Science, Chinese Academy of Sciences, Beijing 100190, China;
		3. School of Mathematical Sciences, University of Chinese Academy of Sciences, Beijing 100049, China.}
	\email{wyd@math.ac.cn}
	
	\author{Guodong Wei}
	\address{School of Mathematics (Zhuhai), Sun Yat-sen University, Zhuhai, Guangdong, 519082, P. R. China}
	\email{weigd3@mail.sysu.edu.cn}
	
	
	
	

	\begin{abstract}
		In this paper, we use the Nash-Moser iteration method to study the local and global behaviors of positive solutions to the nonlinear elliptic equation $\Delta_pv +av^{q}=0$ defined on a complete Riemannian manifolds $(M,g)$ where $p>1$, $a$ and $q$ are constants and $\Delta_p(v)=\di(|\nabla v|^{p-2}\nabla v)$ is the $p$-Laplace operator. Under some assumptions on $a$, $p$ and $q$, we derive gradient estimates and Liouville type theorems for such positive solutions.
	\end{abstract}
	\maketitle

	\section{Introduction}
	Gradient estimate is a fundamental and powerful technique in the study of partial differential equations on Riemannian manifolds. For instance, one can use gradient estimates to derive Harnack inequalities, to deduce Liouville type theorems, to study the geometry of manifolds, etc. Many mathematicians pay attention to the study on this topic (see for example,  \cite{MR615628, MR834612,  peng2020gradient,MR3275651, MR2880214, MR431040, MR3866881} and the references there in).
	
	In this paper we are concerned with the following equation defined on a complete Riemannian manifold $(M,g)$,
	\begin{align}\label{equ0}
		\Delta_pv +av^{q}=0,
	\end{align}
	where $p>1$, $a$, $q\in\R$ are constants and $\Delta_p(v)=\di(|\nabla v|^{p-2}\nabla v)$ is the $p$-Laplace operator.

	When $q\neq p-1$ and  $a>0$, the constant $a$ can be absorbed by a dilation transformation, then equation (\ref{equ0}) reduces to the classical Lane–Emden–Fowler equation
	\begin{align}\label{equa:1.2}
		\Delta_pv +v^{q}=0,
	\end{align}
	which has been widely studied in the literature (see for example,\cite{MR1004713, MR1134481, MR982351, MR1121147, MR615628, MR829846, MR1946918}).
	In particular, Serrin-Zou in \cite{MR1946918} showed that if $1<p<n$, then equation \eqref{equa:1.2} defined on $\R^n$ admits no positive solution if and only if
	$$ 0<q<\frac{np}{n-p}-1.
	$$

	Next, we  pay our main attentions on the equation \eqref{equ0}  defined on a Riemannian manifold. When $p=2$, and $q=2^*$, equation \eqref{equ0} has a deep relationship with the Yamabe problem (see \cite{MR788292, MR929283, MR931204}). When $M$ is the two sphere, equation \eqref{equ0} is also closely relevant to the stationary solutions to Euler's equation on $\mathbb{S}^2$ (see for example, \cite{MR4182319, MR4444081}).
	
	Now, we will focus on the Liouville type results for equation \eqref{equ0} on Riemannian manifolds. When $a = 0$,
	equation (\ref{equ0}) becomes the $p$-Laplace equation
	\begin{align}\label{plaplace}
		\Delta_pv=0, \quad p>1.
	\end{align}
	The celebrated Cheng-Yau’s gradient estimate for harmonic functions showed
	that when $p=2$, any solution bounded from above or below to \eqref{plaplace} is a constant (\cite{MR385749}) provided the Ricci curvature of the manifold is nonnegative.
	Kotschwar-Ni (\cite{MR2518892}) proved that any positive $p$-harmonic function on a complete Riemannian manifold with nonnegative sectional curvature is a constant. Later, Wang-Zhang (\cite{MR2880214}) initially applied the Nash-Moser iteration technique to study the gradient estimates and proved that any positive solution to (\ref{plaplace}) is a constant provided the Ricci curvature of the manifold is nonnegative. Wang-Zhang's results only assumed the lower bound of Ricci curvature of $M$, and hence generalized Cheng-Yau's result (\cite{MR385749}) for harmonic function to $p-$ harmonic function ($p>1$) and substantially improved Kotschwar-Ni's results. Compared with Cheng-Yau's method, Wang-Zhang's approach only involves differentiating the distance function once and hence bypasses the Laplacian comparison theorem.
	
	Motivated by Wang-Zhang's method, Zhao-Yang in \cite{MR3866881} studied the gradient estimates of the following weighted $p$-Laplacian Lichnerowicz equation $$\Delta_{p,f}u+au^{\sigma}=0,\ \ p>1\ \text{and}\ \sigma\leq p-1,$$ defined on a Riemannian manifold with $m$-Bakry-Emery Ricci curvature bounded from below. However, their results can not cover Wang-Zhang's results in the case that $f$ is a constant and $a=0$. Very recently, the second named author and third named author (\cite{MR4559367}) also used the Nash-Moser iteration method to show that
	equation \eqref{equ0} where $p=2$ and $a$ is a positive constant admits no positive solutions  if $(M,g)$ has non-negative Ricci curvature and
	$$
	q\in\left(-\infty, \quad \frac{n+1}{n-1}+\frac{2}{\sqrt{n(n-2)}}\right).$$
	
	Instead of assuming the lower boundedness of curvature, it is worthy to point out that Grigor'yan-Sun showed in \cite{MR3225632} that any nonnegative solution to equation \eqref{equ0} with $p=2$ and $a=1$ is identically zero provided the volume growth of geodesic ball with radius $r$ less than $r^{2q/(q-1)}\ln ^{1/(q-1)}r$.  Later, Sun (\cite{MR3336621}) showed the analogous results also hold for \eqref{equ0} with general $p\ (p>1)$.

	Inspired by Wang-Zhang's work (\cite{MR2880214}), we shall use the Nash-Moser iteration method to study the gradient estimate and the Liouville property to the equation \eqref{equ0} with general $p$ defined on a Riemannian manifold. Now, we state our main results.
	
	\begin{thm}\label{thm1}
		Let $(M,g)$ be an $n$-dim($n>2$) complete manifold with $\ric_g\geq-(n-1)\kappa g$, where $\kappa$ is a non-negative constant. Assume $v$ is a positive solution to equation \eqref{equ0} on the geodesic ball $B(o,R)\subset M$. If the constants $a,q$ and $p\ (p>1)$ satisfy one of the following two conditions,
		\begin{enumerate}
			\item
			\begin{align}\label{cond1}a\left(\frac{n+1}{n-1}-\frac{q}{p-1}\right)\geq0;\end{align}
			\item
			\begin{align}\label{cond2}
				p-1<q<\frac{n+3}{n-1}(p-1).
			\end{align}
			
		\end{enumerate}
		Then there holds
		$$
		\sup_{B_{\frac{R}{2}}(o)} \frac{|\nabla v|^2}{v^2}\leq c(n,p,q )\frac{(1+\sqrt\kappa R)^2}{R^2}.
		$$
	\end{thm}
	
	Note that when $a =0$, one can easily derive Wang-Zhang's gradient estimates (\cite{MR2880214}) from the case (1) in \thmref{thm1}. Moreover, compared with the previous work (\cite{MR3866881}), our result largely extend the range of $q$ for which the gradient estimate holds for \eqref{equ0}.
	
	By carefully analyzing the conditions \eqref{cond1} and \eqref{cond2}  in \thmref{thm1}, the following result holds.
	\begin{cor}\label{cor2}Let $(M,g)$ be an $n$-dim($n>2$) complete manifold with
		$
		\ric_g\geq-(n-1)\kappa g
		$,
		where $\kappa$ is a non-negative constant. Assume $v$ is a positive solution to equation \eqref{equ0} on the geodesic ball $B(o,R)\subset M$. If
		\begin{align*}
			a>0  \quad \text{ and } \quad q<\frac{n+3}{n-1}(p-1),
		\end{align*}
		or
		\begin{align*}
			a<0 \quad \text{ and } \quad q>p-1,
		\end{align*}
		then
		$$
		\sup_{B_{\frac{R}{2}}(o)} \frac{|\nabla v|^2}{v^2}\leq c(n,p,q)\frac{(1+\sqrt\kappa R)^2}{R^2}.
		$$
	\end{cor}
	
	When the Ricci curvature of $(M,g)$ is non-negative, we can obtain the corresponding Liouville property of the equation \eqref{equ0}.
	\begin{thm}\label{thm2}
		Let $(M,g)$ be a complete non-compact Riemannian manifold with non-negative Ricci curvature. If $a, \ p$ and $q$ satisfy one of the conditions given in \thmref{thm1}, then equation (\ref{equ0}) admits no  positive solutions.
	\end{thm}

	\begin{rem}
		Here we want to give several remarks of the above Liouville type result.
		\begin{enumerate}
			\item There is no extra restriction on $p$ in our result.
			\item In \cite{MR1946918}, Serrin-Zou showed that if $1<p<n$, then equation \eqref{equa:1.2} defined on $\R^n$ admits no positive solution if and only if $ 0<q<np/(n-p)-1$. Combining this with our result, it seems that the optimal range of $q$ for which the Liouville type result holds for \eqref{equa:1.2} defined on a complete Riemannian manifold of nonnegative Ricci curvature with $1<p<n$ is $\left(-\infty,np/(n-p)-1\right)$.  Direct computation shows that
			$$\frac{n+3}{n-1}(p-1)<\frac{np}{n-p}-1.$$
			Thus, this problem certainly deserves further study.
			
			\item When $a>0$ and $p=2$, by \thmref{thm2}, we deduce that for
			$
			q\in\left(-\infty, (n+3)/(n-1)\right)
			$,
			equation(\ref{equ0}) has no positive solutions. Since
			$$
			\frac{n+3}{n-1}>\frac{n+1}{n-1}+\frac{2}{\sqrt{n(n-1)}},
			$$
			\thmref{thm2} improved Wang-Wei's main results in \cite{MR4559367}.
	\end{enumerate}\end{rem}

	\begin{thm}\label{thm1.6}
		Let $(M,g)$ be a complete non-compact Riemannian manifold with $
		\ric_g\geq-(n-1)\kappa g,
		$
		where $\kappa$ is a non-negetive constant. Suppose $v$ is a positive solution of equation \eqref{equ0} with the constants $a,p$ and $q$ satisfy \eqref{cond1} or \eqref{cond2}. Fix $x_0\in M$, then for any $x\in M$, there holds
		$$
		v(x_0)e^{-c(n,p,q)\sqrt{\kappa}d(x,x_0)}\leq    v(x)\leq v(x_0)e^{c(n,p,q)\sqrt{\kappa}d(x,x_0)},
		$$
		where $d(x_0,x)$ is the geodesic distance between $x_0$ and $x$.
	\end{thm}
	
	We notice that just one week ago, Huang-Guo-Guo posted a preprint on arxiv (see \cite{huang2023gradient}) in where they also study the gradient estimates of equation (\ref{equ0}). We point out here that their results are much different from ours. First, their Theorem 1 and Theorem 3 are all weaker than ours since in our gradient estimate and Liouville type theorem, the range of $q$ is larger. What's more, their Theorem 1 needs a restriction on $p$ and their Theorem 3 does not have a detailed proof; Second,  it is mysterious that their gradient estimate can not completely cover Wang-Zhang's gradient estimate for $p$-harmonic functions (see \cite{MR2880214}).
	
	The rest of  our paper is organized as follows. In section 2, we will give a meticulous estimate of  $\mathcal{L} \left(|\nabla \log v|^{2\alpha}\right)$ (see \eqref{linea} for the explicit definition of the operator $\mathcal{L}$). We will see in the proofs of our main results that by choosing an appropriate $\alpha$, we can largely improve the range of $q$ for which the gradient estimate holds for equation \eqref{equ0}. This is the main innovative point in our article. We shall also recall  Saloff-Coste's Sobolev embedding theorem (\cite[Theorem 3.1]{saloff1992uniformly}) in this section that will play a key role in the iteration process. In section 3, we will carefully use the Moser iteration to provide the proofs of the main results in this paper.
	\section{Preliminaries}
	Throughout this paper, we denote $(M,g)$ an $n$-dim Riemannian manifold, and $\nabla$ the corresponding Levi-Civita connection. For any function $\varphi\in C^1(M)$, we denote $\nabla \varphi\in \Gamma(T^*M)$ by $\nabla \varphi(X)=\nabla_X\varphi$. We denote the volume form $\vol=\sqrt{\det(g_{ij})}d x_1\wedge\ldots\wedge dx_n$ where $(x_1,\ldots, x_n)$ is a local coordinates, and for simplicity we may omit the volume form of integral over $M$.
	
	The $p$-Laplace operator is defined by
	$$
	\Delta_pu=\di\left(|\nabla u|^{p-2}\nabla u\right).
	$$
	The solution of $p$-Laplace equation $\Delta_pu=0$ is the critical point of the energy functional
	$$
	E(u)=\int_M|\nabla u|^p.
	$$

	\begin{defn}\label{def1}
		$v$ is said to be a (weak) solution of equation (\ref{equ0}) on a region $\Omega\subset M$, if $v\in C^1(\Omega)\cap W^{1,p}_{loc}(\Omega)$
		and for all $\psi\in W^{1,p}_0(\Omega)$, we have
		\begin{align*}
			-\int_\Omega|\nabla v|^{p-2}\la\nabla v,\nabla\psi\ra +\int_\Omega av^q\psi=0.
		\end{align*}
	\end{defn}
	It is worth mentioning that any solution $v$ of equation (\ref{equ0}) satisfies $v\in W^{2,2}_{loc}(\Omega)$ and $v\in C^{1,\beta}(\Omega)$ for some $\beta\in(0,1)$(for example, see \cite{MR0709038, MR0727034,MR0474389}).
	
	Next, we recall the  Saloff-Coste's Sobolev inequalities (see \cite[Theorem 3.1]{saloff1992uniformly}) which shall play an key role in our proof of the main theorem.
	\begin{lem}[\cite{saloff1992uniformly}]\label{salof}
		Let $(M,g)$ be a complete manifold with $Ric\geq-(n-1)\kappa$. For $n>2$, there exists a positive constant $C_n$ depending only on $n$, such that for all $B\subset M$ of radius R and  volume $V$ we have for $f\in C^{\infty}_0(B)$
		$$
		\|f\|_{L^{\frac{2n}{n-2}}}^2\leq e^{C_n(1+\sqrt{\kappa}R)}V^{-\frac{2}{n}}R^2\left(\int|\nabla f|^2+R^{-2}f^2\right).
		$$
		For $n=2$, the above inequality holds with $n$ replaced by any fixed $n'>2$.
	\end{lem}

	By a logarithmic transformation $u = -(p-1)\log v$, equation (\ref{equ0}) becomes
	\begin{align}\label{equ21}
		\Delta_pu -|\nabla u|^p-be^{cu}=0,
	\end{align}
	where
	\begin{align*}
		b = a(p-1)^{p-1},\quad  c = \frac{p-q-1}{p-1}.
	\end{align*}
	Now we consider the linearisation operator $\Lie$ of $p$-Laplace operator:
	\begin{align}\label{linea}
		\Lie(\psi)=\di\left(f^{p/2-1}A(\nabla \psi)\right),
	\end{align}
	where $f = |\nabla u|^2$, and
	\begin{align}\label{defofA}
		A(\nabla\psi) = \nabla\psi+(p-2)f^{-1}\la\nabla \psi,\nabla u\ra\nabla u.
	\end{align}
	We first derive an useful expression of $\mathcal L(f^\alpha)$ for any $\alpha>0$.
	
	\begin{lem}
		For any $\alpha >0$, the equality
		\begin{align}
			\label{bochner1}
			\begin{split}
				\mathcal{L} (f^{\alpha}) =
				&
				\alpha\left(\alpha+\frac{p}{2}-2\right)f^{\alpha+\frac{p}{2}-3}|\nabla f|^2
				+
				2\alpha f^{\alpha+\frac{p}{2}-2} \left(|\nabla\nabla u|^2 + \ric(\nabla u,\nabla u) \right)
				\\
				&
				+
				\alpha(p-2)(\alpha-1)f^{\alpha+\frac{p}{2}-4}\langle\nabla f,\nabla u\rangle^2
				+2\alpha f^{\alpha-1}\langle\nabla\Delta_p u,\nabla u\rangle
			\end{split}
		\end{align}
		holds point-wisely in $\{x:f(x)>0\}$.
	\end{lem}
	
	\begin{proof}
		By the definition of $A$ in (\ref{defofA}), we have
		$$
		A\Big(\nabla (f^{\alpha})\Big) = \alpha f^{\alpha-1}\nabla f + \alpha(p-2)f^{\alpha-2}\langle\nabla f,\nabla u\rangle \nabla u = \alpha f^{\alpha-1}A(\nabla f).
		$$
		Hence
		\begin{align*}
			\mathcal{L} (f^{\alpha}) &
			= \alpha \text{div}\Big( f^{\alpha-1}f^{\frac{p}{2}-1}A(\nabla f)\Big)
			= \alpha \Big\langle\nabla( f^{\alpha-1}),f^{\frac{p}{2}-1}A(\nabla f)\Big\rangle + \alpha f^{\alpha-1}\mathcal{L} (f).
		\end{align*}
		
		Direction computation shows that
		\begin{align}
			\label{equ2.6}
			\alpha \Big\langle\nabla( f^{\alpha-1}),f^{\frac{p}{2}-1}A(\nabla f)\Big\rangle
			=&
			\Big\langle \alpha(\alpha-1)f^{\alpha-2}\nabla f, f^{\frac{p}{2}-1}\nabla f + (p-2)f^{\frac{p}{2}-2}\langle\nabla f,\nabla u\rangle\nabla u\Big\rangle,
		\end{align}
		and
		\begin{align}
			\label{equ2.7}
			\begin{split}
				\alpha f^{\alpha-1}\mathcal{L} (f)
				=
				& \alpha f^{\alpha-1}\Big(\left(\frac{p}{2}-1\right)f^{\frac{p}{2}-2}|\nabla f|^2 + f^{\frac{p}{2}-1}\Delta f + (p-2)\left(\frac{p}{2}-2\right)f^{\frac{p}{2}-3}\langle\nabla f,\nabla u\rangle^2
				\\
				&
				+ (p-2)f^{\frac{p}{2}-2}\langle\nabla\langle\nabla f,\nabla u\rangle,\nabla u\rangle + (p-2)f^{\frac{p}{2}-2}\langle\nabla f,\nabla u\rangle\Delta u\Big).
			\end{split}
		\end{align}
		We combine (\ref{equ2.6}) and (\ref{equ2.7}) to obtain
		\begin{align}
			\label{equ2.9}
			\begin{split}
				\Lie(f^\alpha)=
				&
				\alpha\left(\alpha+\frac{p}{2}-2\right)f^{\alpha+\frac{p}{2}-3}|\nabla f|^2
				+
				\alpha f^{\alpha+\frac{p}{2}-2}\Delta f
				\\
				&
				+
				\alpha(p-2)\left(\alpha+\frac{p}{2}-3\right)f^{\alpha+\frac{p}{2}-4}\langle\nabla f,\nabla u\rangle^2\\
				& + \alpha(p-2)f^{\alpha+\frac{p}{2}-3}\langle\nabla\langle\nabla f,\nabla u\rangle,\nabla u\rangle
				+
				\alpha(p-2)f^{\alpha+\frac{p}{2}-3}\langle\nabla f,\nabla u\rangle\Delta u.
			\end{split}
		\end{align}
		
		On the other hand, by the definition of the $p$-Laplacian we have
		\begin{align*}
			\langle\nabla\Delta_p u,\nabla u\rangle &= \left(\frac{p}{2}-1\right)\left(\frac{p}{2}-2\right)f^{\frac{p}{2}-3}\langle\nabla f,\nabla u\rangle^2 + \left(\frac{p}{2}-1\right)f^{\frac{p}{2}-2}\langle\nabla\langle\nabla f,\nabla u\rangle,\nabla u\rangle\\
			&\quad+ \left(\frac{p}{2}-1\right)f^{\frac{p}{2}-2}\langle\nabla f,\nabla u\rangle\Delta u + f^{\frac{p}{2}-1}\langle\nabla\Delta u,\nabla u\rangle.
		\end{align*}
		Thus, the last term of the right hand side of (\ref{equ2.9}) can be rewritten as
		\begin{align}\label{last}
			\begin{split}
				\alpha(p-2)f^{\alpha+\frac{p}{2}-3}\langle\nabla f,\nabla u\rangle\Delta u
				&
				=
				2\alpha f^{\alpha-1}\langle\nabla\Delta_p u,\nabla u\rangle
				- 2\alpha f^{\alpha+\frac{p}{2}-2}\langle\nabla\Delta u,\nabla u\rangle
				\\
				&
				\quad- \alpha(p-2)\left(\frac{p}{2}-2\right)f^{\alpha+\frac{p}{2}-4}\langle\nabla f,\nabla u\rangle^2
				\\
				&
				\quad- \alpha(p-2)f^{\alpha+\frac{p}{2}-3}\langle\nabla\langle\nabla f,\nabla u\rangle,\nabla u\rangle.
			\end{split}
		\end{align}
		By (\ref{last}) and the following Bochner formula $$\frac{1}{2}\Delta f = |\nabla\nabla u|^2 + \ric(\nabla u,\nabla u) + \langle\nabla\Delta u,\nabla u\rangle,$$
		we have
		\begin{align*}
			\mathcal{L}(f^{\alpha}) =
			&
			\alpha\left(\alpha+\frac{p}{2}-2\right)f^{\alpha+\frac{p}{2}-3}|\nabla f|^2
			+
			2\alpha f^{\alpha+\frac{p}{2}-2} \left(|\nabla\nabla u|^2 + \ric(\nabla u,\nabla u) \right)
			\\
			&
			+
			\alpha(p-2)(\alpha-1)f^{\alpha+\frac{p}{2}-4}\langle\nabla f,\nabla u\rangle^2
			+2\alpha f^{\alpha-1}\langle\nabla\Delta_p u,\nabla u\rangle
			.
		\end{align*}
		Thus, we finish the proof of the lemma.
	\end{proof}

	\section{Proof of main theorem}
	
	\subsection{Estimate for the linearisation operator of $p$-Laplace operator}
	We first need to prove a pointwise estimate for $\mathcal L(f^\alpha)$.
	\begin{lem}\label{lem31}\label{linear}
		Let $u$ be a solution of equation (\ref{equ21}) on $(M,g)$ with $\mathrm{Ric}\geq -(n-1)\kappa$.  Denote $f=|\nabla u|^2$ and $a_1=\left|p-\frac{2(p-1)}{n-1}\right|$. The followings hold:
		\begin{enumerate}
			\item If $$a\left(\frac{n+1}{n-1}-\frac{q}{p-1}\right)\geq0,$$ then for any $\alpha\geq 1$, the inequality
			\begin{align*}
				\mathcal L(f^\alpha)\geq&
				\frac{  2\alpha f^{\alpha+\frac{p}{2}}}{n-1}- 2(n-1)\alpha \kappa f^{\alpha+\frac{p}{2}-1}
				-\alpha a_1 f^{\alpha+\frac{p}{2}-\frac{3}{2} }|\nabla f| .
			\end{align*}
			holds point-wisely in $\{x:f(x)>0\}$.
			\item If
			\begin{align*}
				p-1<q<\frac{n+3}{n-1}(p-1),
			\end{align*}
			then there exists some $\alpha_0(p,q,n)>0$ such that for any $\alpha\geq\alpha_0$, the inequality
			\begin{align*}
				\mathcal L(f^\alpha)\geq&
				2\alpha \beta_{n,p,q,\alpha}  f^{\alpha+\frac{p}{2}} - 2\alpha(n-1) \kappa f^{\alpha+\frac{p}{2}-1}
				-\alpha a_1 f^{\alpha+\frac{p}{2}-\frac{3}{2} }|\nabla f|
			\end{align*}
			holds point-wisely in $\{x:f(x)>0\}$, where
			\begin{align*}
				\beta_{n,p,q,\alpha}=\frac{1}{n-1}-\left(\frac{n+1}{n-1}-\frac{q}{p-1}\right)^2\frac{ (2\alpha-1)(n-1)+p-1  }{4(2\alpha-1) }
				>0.
			\end{align*}

		\end{enumerate}
	\end{lem}
	
	\begin{proof}
		Let $\{e_1,e_2,\ldots, e_n\}$ be  an orthonormal frame of $TM$ on a domain with $f\neq 0$ such that $e_1=\frac{\nabla u}{|\nabla u|}$. We have $u_1 = f^{1/2}$  and
		\begin{align}\label{equ:3.1}
			u_{11} = \frac{1}{2}f^{-1/2}f_1 = \frac{1}{2}f^{-1}\la\nabla u,\nabla f\ra.
		\end{align}
		In this case, $\Delta_p u$ has the following expression (see \cite{MR2518892, MR2880214}),
		\begin{align*}
			\Delta_p u
			=&f^{\frac{p}{2}-1}\left((p-1)u_{11}+\sum_{i=2}^nu_{ii}\right).
		\end{align*}
		Substituting the above equality into equation (\ref{equ21}), we obtain:
		\begin{align}\label{equ3.3}
			(p-1)u_{11}+\sum_{i=2}^nu_{ii}=f+be^{cu}f^{1-\frac{p}{2}}.
		\end{align}
		Using the fact $u_1 = f^{1/2}$ again yields
		\begin{align}
			\label{equ3.4}
			|\nabla f|^2/f=4\sum_{i=1}^n u_{1i}^2 .
		\end{align}
		By Cauchy inequality, we arrive at
		\begin{align}
			\label{equ3.5}
			|\nabla\nabla u|^2
			\geq& \sum_{i=1}^nu_{1i}^2 +
			\sum_{i=2}u_{ii}^2
			\geq \frac{|\nabla f|^2}{4f} +\frac{1}{n-1}\left(\sum_{i=2}u_{ii}\right)^2.
		\end{align}
		It follows from \eqref{equ21} that
		\begin{align}\label{3.6}
			\langle\nabla\Delta_p u,\nabla u\rangle =  pf^{\frac{p}{2} }u_{11}+bce^{cu}f .
		\end{align}
		Substituting (\ref{equ:3.1}), (\ref{equ3.4}), (\ref{equ3.5}) and (\ref{3.6}) into (\ref{bochner1}),  we have
		\begin{align}
			\label{equ3.60}
			\begin{split}
				\frac{f^{2-\alpha-\frac{p}{2}}}{2\alpha} \mathcal{L} \left(f^{\alpha}\right)
				\geq
				&
				\frac{1}{2} \left(\alpha+\frac{p-3}{2}\right) \frac{|\nabla f|^2}{f}
				+\frac{1}{n-1}\left(\sum_{i=2}u_{ii}\right)^2+ \ric(\nabla u, \nabla u)
				\\
				&
				+
				2(p-2)(\alpha-1)u_{11}^2
				+ f^{1-\frac{p}{2}}\left(pf^{\frac{p}{2} }u_{11}+bce^{cu}f\right)
				.
			\end{split}
		\end{align}
		Furthermore, noting that $|\nabla f|^2/f\geq 4u_{11}^2$ which follows from \eqref{equ3.4}, and the fact $\alpha+\frac{p-3}{2}>0$ since $\alpha\geq 1$ and $ p>1$, we can infer from \eqref{equ3.60} that
		
		\begin{align}
			\label{equ3.6}
			\begin{split}
				\frac{f^{2-\alpha-\frac{p}{2}}}{2\alpha} \mathcal{L} \left(f^{\alpha}\right)
				\geq
				&
				2 \left(\alpha+\frac{p-3}{2}\right) u_{11}^2
				+\frac{1}{n-1}\left(\sum_{i=2}u_{ii}\right)^2+ \ric(\nabla u, \nabla u)
				\\
				&
				+
				2(p-2)(\alpha-1)u_{11}^2
				+ f^{1-\frac{p}{2}}\left(pf^{\frac{p}{2} }u_{11}+bce^{cu}f\right).
			\end{split}
		\end{align}

		By (\ref{equ3.3}), we have
		\begin{align*}
			\left(\sum_{i=2}u_{ii}\right)^2
			=&   \left(f+be^{cu}f^{1-\frac{p}{2}}-(p-1)u_{11}\right)^2\\
			=&
			f^2+\left(be^{cu}f^{1-\frac{p}{2}}-(p-1)u_{11}\right)^2+2be^{cu}f^{2-\frac{p}{2}}-2f(p-1)u_{11}.
		\end{align*}

		Substituting the above inequality into (\ref{equ3.6}) yields
		\begin{align}
			\label{equ3.7}
			\begin{split}
				\frac{f^{2-\alpha-\frac{p}{2}}}{2\alpha} \mathcal{L}\left(f^\alpha\right)
				\geq
				&
				(p-1)(2\alpha-1)u_{11}^2-(n-1)\kappa f+   \left(p-\frac{2(p-1)}{n-1}\right)f u_{11}+ \frac{f^2}{n-1}
				\\
				&
				+b\left(c+\frac{2}{n-1}\right)e^{cu}f^{2-\frac{p}{2}}
				+\frac{1}{n-1}   \left(be^{cu}f^{1-\frac{p}{2}}-(p-1)u_{11}\right)^2
				.
			\end{split}
		\end{align}
		
		If we denote $a_1 = \left|p-\frac{2(p-1)}{n-1}\right|$,  by (\ref{equ:3.1}) we have
		\begin{align*}
			2\left(p-\frac{2(p-1)}{n-1}\right)f u_{11}\geq -a_1f^{\frac{ 1}{2} }|\nabla f|.
		\end{align*}
		Substituting $a_1$ into (\ref{equ3.7}), we have
		\begin{align}
			\label{equa:3.7}
			\begin{split}
				\frac{f^{2-\alpha-\frac{p}{2}}}{2\alpha} \mathcal{L}(f^\alpha)
				\geq
				&
				(p-1)(2\alpha-1)u_{11}^2-(n-1)\kappa f-\frac{a_1}{2}f^{\frac{ 1}{2} }|\nabla f|+ \frac{f^2}{n-1}
				\\
				&
				+b\left(c+\frac{2}{n-1}\right)e^{cu}f^{2-\frac{p}{2}}
				+\frac{1}{n-1}\left(  be^{cu}f^{1-\frac{p}{2}}-(p-1)u_{11}\right)^2.
			\end{split}
		\end{align}
		
		\textbf{Case \Rmnum{1}:} the constants $a$, $p$ and $q$ satisfy $$a\left(\frac{n+1}{n-1}-\frac{q}{p-1}\right)\geq0.$$
		
		For this case we have
		$$
		be^{cu}f\left(c+\frac{2}{n-1}\right)=a(p-1)^{p-1}e^{cu}f\left(\frac{n+1}{n-1}-\frac{q}{p-1}\right)\geq 0.
		$$
		Since $\alpha\geq 1$, by discarding some non-negative terms in (\ref{equa:3.7}) we obtain
		\begin{align*}
			\mathcal L(f^\alpha)\geq&  2\alpha f^{\alpha+\frac{p}{2}-2}\left(\frac{  f^{2}}{n-1}- (n-1)\kappa f -\frac{a_1}{2}f^{\frac{1}{2} }|\nabla f|\right),
		\end{align*}
		which is just the inequality in the first case of \lemref{lem31}.
		
		By expanding the last term of the right hand side of (\ref{equa:3.7}), we obtain
		\begin{align}\label{2.12}
			\begin{split}
				\frac{f^{2-\alpha-\frac{p}{2}}}{2\alpha} \mathcal{L} (f^{\alpha})\geq &(p-1)\left(2\alpha-1+\frac{p-1}{n-1}\right)u_{11}^2 - (n-1)\kappa f\\
				&+b\left(\frac{n+1}{n-1}-\frac{q}{p-1}\right)e^{cu}f^{2-\frac{p}{2}}-\frac{a_1}{2}f^{\frac{ 1}{2} }|\nabla f|+\frac{f^2}{n-1}\\
				&+\frac{1}{n-1}\left(  b^2e^{2cu}f^{2-p}-2(p-1)be^{cu}f^{1-\frac{p}{2}}u_{11}\right).
			\end{split}
		\end{align}

		\textbf{Case \Rmnum{2} :} the constants $a$, $p$ and $q$ satisfy
		\begin{align*}
			p-1<q<\frac{n+3}{n-1}(p-1).
		\end{align*}
		
		In the present situation, the condition is equivalent to
		\begin{align*}
			\frac{1}{n-1}-\frac{1}{4}\left(\frac{n+1}{n-1}-\frac{q}{p-1}\right)^2>0.
		\end{align*}
		This implies
		\begin{align*}
			\lim_{\alpha\to\infty}\frac{1}{n-1}-\left(\frac{n+1}{n-1}-\frac{q}{p-1}\right)^2\frac{ (2\alpha-1)(n-1)+p-1  }{4(2\alpha-1) }
			>0.
		\end{align*}
		Thus, we can choose $\alpha_0=\alpha_0(n,p,q)$ large enough such that, for any $\alpha\geq\alpha_0$ there holds true
		\begin{align}\label{defofdelta}
			\beta_{n,p,q,\alpha}=\frac{1}{n-1}-\left(\frac{n+1}{n-1}-\frac{q}{p-1}\right)^2\frac{ (2\alpha-1)(n-1)+p-1  }{4(2\alpha-1) }
			>0.
		\end{align}

		On the other hand, by using the inequality $a^2-2ab\geq -b^2$ we have
		\begin{align}
			\label{equa:2.13}
			\begin{split}
				&(p-1)\left(2\alpha-1+\frac{p-1}{n-1}\right)u_{11}^2-2
				\frac{(p-1)}{n-1}be^{cu}f^{1-\frac{p}{2}}u_{11}\\
				\geq &-\frac{(p-1)b^2e^{2cu}f^{2-p}}{((2\alpha-1)(n-1)+p-1)(n-1)} .
			\end{split}
		\end{align}
		Combining (\ref{2.12}) and (\ref{equa:2.13}) yields
		\begin{align}
			\label{equa:2.14}
			\begin{split}
				\frac{f^{2-\alpha-\frac{p}{2}}}{2\alpha} \mathcal{L}(f^\alpha)
				\geq &\frac{(2\alpha-1) b^2e^{2cu}f^{2-p}}{(2\alpha-1)(n-1)+p-1}
				- \kappa(n-1)f -\frac{a_1}{2}f^{\frac{1}{2}}|\nabla f|\\
				& + b\left(\frac{n+1}{n-1}-\frac{q}{p-1}\right)e^{cu}f^{2-\frac{p}{2}}+\frac{f^2}{n-1}.
			\end{split}
		\end{align}
		Applying the inequality $a^2+2ab\geq -b^2$ again, we have
		\begin{align}
			\label{equa:3.15}
			\begin{split}
				&\frac{(2\alpha-1) b^2e^{2cu}f^{2-p}}{ (2\alpha-1)(n-1)+p-1  }
				+ b\left(\frac{n+1}{n-1}-\frac{q}{p-1}\right)e^{cu}f^{2-\frac{p}{2}}\\
				\geq& -\left(\frac{n+1}{n-1}-\frac{q}{p-1}\right)^2\frac{ (2\alpha-1)(n-1)+p-1  }{4(2\alpha-1) }f^2
			\end{split}
		\end{align}
		Substituting (\ref{equa:3.15}) into (\ref{equa:2.14}), we arrive at
		\begin{align*}
			\frac{f^{2-\alpha-\frac{p}{2}}}{2\alpha} \mathcal{L} (f^\alpha)
			\geq
			&
			\left(\frac{1}{n-1}-\left(\frac{n+1}{n-1}-\frac{q}{p-1}\right)^2\frac{ (2\alpha-1)(n-1)+p-1}{4(2\alpha-1) }\right)f^2
			\\
			& - (n-1)\kappa f-\frac{a_1}{2}f^{\frac{1}{2}}|\nabla f|.
		\end{align*}
		Hence,
		\begin{align*}
			\mathcal L(f^\alpha)
			\geq&   2\beta_{n,p,q,\alpha} \alpha f^{\alpha+\frac{p}{2} } - 2\alpha(n-1)\kappa f^{\alpha+\frac{p}{2}-1}
			-a_1\alpha f^{\alpha+\frac{p}{2}-\frac{3}{2}}|\nabla f| ,
		\end{align*}
		where $\beta_{n,p,q,\alpha}>0$ is defined in (\ref{defofdelta}). Thus, we complete the proof of this lemma.
	\end{proof}

	From now on, we fix
	\begin{align*}
		\alpha=\alpha_0(n,p,q)
	\end{align*}
	and use $a_1$, $a_2$, $\cdots$, to denote constants depending only on $n$, $p$ and $q$. It is worthy to point out that if
	$$
	q\to\frac{n+3}{n-1}(p-1) \quad \text{or}\quad q\to p-1,
	$$
	then $\alpha_0(n,p,q)\to \infty$.
	
	Denote $\beta_{n,p,q} = \beta_{n,p,q,\alpha_0}$. Moreover, from the definition of $\beta_{n,p,q,\alpha}$ (see \eqref{defofdelta}) we can also see easily that
	\begin{align*}
		\beta_{n,p,q} < \frac{1 }{n-1}.
	\end{align*}
	So, we actually have proved that there holds
	\begin{align}
		\label{equa3.14}
		\mathcal L(f^{\alpha_0})\geq
		2\alpha_0 f^{\alpha_0+\frac{p}{2}-2}\left(\beta_{n,p,q }  f^{2 } - 2 (n-1)\kappa f
		-\frac{a_1}{2}  f^{ \frac{1}{2}}|\nabla f|\right),
	\end{align}
	if one of the conditions (1) and (2) in \lemref{linear} is satisfied.

	\subsection{Deducing the main integral inequality}\label{sec3.2}
	Now, we need to establish a key integral inequality of $f=|\nabla u|^2$.
	\begin{lem}
		Let $\Omega = B_R(o)\subset M$ be a geodesic ball. Under the same assumptions as in \lemref{lem31}, we have the following integral inequality
		\begin{align*}
			\begin{split}
				& \beta_{n,p,q} \int_\Omega f^{\alpha_0+\frac{p}{2}+t}\eta^2 +
				\frac{a_3}{ t }e^{-t_0}V^{\frac{2}{n}}R^{-2}\left\|f^{\frac{\alpha_0+t-1}{2}+\frac{p}{4}}\eta\right\|_{L^{\frac{2n}{n-2}}}^2\\
				\leq & a_5t_0^2R^{-2} \int_\Omega f^{\alpha_0+\frac{p}{2}+t-1}\eta^2+\frac{a_4}{t }\int_\Omega f^{\alpha_0+\frac{p}{2}+t-1}|\nabla\eta|^2,
			\end{split}
		\end{align*}
		where $\beta_{n,p,q}$ is given in the above \lemref{lem31}, $a_3$, $a_4$ and $a_5$ depend only on $n$, $p$ and $q$.
	\end{lem}
	
	\begin{proof}
		By choosing $\psi = f_{\epsilon}^{t}\eta^2$ as the test function of \eqref{equa3.14}, where $\eta\in C^{\infty}_0(\Omega,\R)$ is non-negative, $f_\epsilon=(f-\epsilon)^+$ with $\epsilon>0$, and $t>1$ is to be determined later, we can deduce from (\ref{equa3.14}) that
		\begin{align*}
			&-\int_\Omega\la f^{p/2-1}\nabla f^{\alpha_0} +(p-2)f^{p/2-2}\la\nabla f^{\alpha_0},\nabla u\ra\nabla u,\nabla \psi \ra
			\\
			\geq &
			2\beta_{n,p,q }\alpha_0\int_\Omega f^{\alpha_0+\frac{p}{2} }f^{t}_\epsilon\eta^2  -2(n-1)\alpha_0\kappa\int_\Omega f^{\alpha_0+\frac{p}{2} -1}f^{t}_\epsilon\eta^2
			- a_1\alpha_0\int_\Omega f^{\alpha_0+\frac{p-3}{2} }f^{t}_\epsilon|\nabla f|\eta^2.
		\end{align*}
		Hence,
		\begin{align}
			\label{equa:3.17}
			\begin{split}
				&-\int_\Omega \alpha_0 tf^{\alpha_0+\frac{p}{2}-2}f^{t-1}_\epsilon|\nabla f|^2\eta^2
				+
				t\alpha_0(p-2)f^{\alpha_0+\frac{p}{2}-3}f^{t-1}_\epsilon\la\nabla f,\nabla u\ra^2\eta^2
				\\
				&-\int_\Omega2\eta\alpha_0 f^{\alpha_0+\frac{p}{2}-2}f^{t}_\epsilon\la\nabla f,\nabla\eta\ra+2\alpha_0\eta(p-2)f^{\alpha_0+\frac{p}{2}-3}f^{t}_\epsilon\la\nabla f,\nabla u\ra\la\nabla u,  \nabla\eta\ra\\
				\geq & 2\beta_{n,p,q }\alpha_0\int_\Omega f^{\alpha_0+\frac{p}{2} }f^{t}_\epsilon\eta^2  -2(n-1)\alpha_0\kappa\int_\Omega f^{\alpha_0+\frac{p}{2} -1}f^{t}_\epsilon\eta^2- a_1\alpha_0\int_\Omega f^{\alpha_0+\frac{p-3}{2} }f^{t}_\epsilon|\nabla f|\eta^2.
			\end{split}
		\end{align}
		
		Next, we need to use the following two inequalities
		\begin{align}\label{2.24}
			f^{t-1}_\epsilon |\nabla f|^2 +(p-2)f^{t-1}_\epsilon f^{-1}\la\nabla f,\nabla u\ra^2\geq a_2 f^{t-1}_\epsilon|\nabla f|^2
		\end{align}
		where $a_2 = \min\{1, p-1\}$, and
		\begin{align}\label{2.25}
			f^{t}_\epsilon\la\nabla f,\nabla\eta\ra+ (p-2)f^{t}_\epsilon f^{-1}\la\nabla f,\nabla u\ra\la\nabla u,  \nabla\eta\ra\geq -(p+1)  f^{t}_\epsilon|\nabla f||\nabla\eta|.
		\end{align}
		Now, substituting (\ref{2.24}) and (\ref{2.25}) into (\ref{equa:3.17}), dividing both sides by $\alpha_0$ and letting $\epsilon\to0$, we can obtain
		\begin{align}\label{2.26}
			\begin{split}
				&2\beta_{n,p,q }\int_\Omega f^{\alpha_0+\frac{p}{2}+t}\eta^2
				+
				a_2  t\int_\Omega f^{\alpha_0+\frac{p}{2}+t-3}|\nabla f|^2\eta^2\\
				\leq &
				2(n-1) \kappa\int_\Omega f^{\alpha_0+\frac{p}{2}+t-1}\eta^2
				+ a_1 \int_\Omega f^{\alpha_0+\frac{p-3}{2}+t }|\nabla f|\eta^2
				\\
				&
				+2 (p+1)\int_\Omega  f^{\alpha_0+\frac{p}{2}+t-2}|\nabla f||\nabla\eta|\eta.
			\end{split}
		\end{align}
		Since $u\in W^{2,2}_{loc}(\Omega)\cap C^{1,\beta}(\Omega)$, we have $f\in C^\beta(\Omega)$ and $|\nabla f|\in L^2_{loc}$, and hence the integrals in the above make sense.
		
		By Cauchy-inequality, we have
		\begin{align}
			\label{2.27}
			a_1 f^{\alpha_0+\frac{p-3}{2}+t }|\nabla f|\eta^2\leq
			\frac{a_2t}{4}    f^{\alpha_0+\frac{p}{2}+t-3}|\nabla f|^2\eta^2
			+\frac{ a_1^2}{a_2t} f^{\alpha_0+\frac{p}{2}+t}\eta^2,
		\end{align}
		and
		\begin{align}\label{new}
			2(p+1)f^{\alpha_0+\frac{p}{2}+t-2}|\nabla f||\nabla\eta|\eta\leq
			\frac{a_2t}{4}    f^{\alpha_0+\frac{p}{2}+t-3}|\nabla f|^2\eta^2
			+\frac{4(p+1)^2 }{a_2t} f^{\alpha_0+\frac{p}{2}+t-1}|\nabla \eta|^2.
		\end{align}
		Now we choose $t$ large enough such that
		\begin{align}\label{2.28}
			\frac{a_1^2}{a_2t}\leq  \beta_{n,p,q}.
		\end{align}
		Then, it follows from \eqref{2.26}, \eqref{2.27}, \eqref{new} and \eqref{2.28} that
		\begin{align}
			\label{2.29}
			\begin{split}
				& \beta_{n,p,q }\int_\Omega f^{\alpha_0+\frac{p}{2}+t}\eta^2
				+
				\frac{a_2  t}{2}\int_\Omega f^{\alpha_0+\frac{p}{2}+t-3}|\nabla f|^2\eta^2 \\
				\leq  &
				2(n-1) \kappa\int_\Omega f^{\alpha_0+\frac{p}{2}+t-1}\eta^2
				+\frac{4(p+1)^2 }{a_2t}  \int_\Omega f^{\alpha_0+\frac{p}{2}+t-1}|\nabla \eta|^2.
			\end{split}
		\end{align}
		
		On the other hand, we have
		\begin{align}\label{2.30}
			\begin{split}
				\frac{1}{2}\left|\nabla \left(f^{\frac{\alpha_0+t-1}{2}+\frac{p}{4} }\eta \right)\right|^2\leq & \left|\nabla f^{\frac{\alpha_0+t-1}{2}+\frac{p}{4} }\right|^2\eta^2 +f^{\alpha_0+t-1+\frac{p}{2}}|\nabla\eta|^2
				\\
				=&\frac{(2\alpha_0+2t+p-2)^2}{16}f^{\alpha_0+t+\frac{p}{2}-3}|\nabla f |^2\eta^2 +f^{\alpha_0+t-1+\frac{p}{2}}|\nabla\eta|^2  .
			\end{split}
		\end{align}
		Substituting (\ref{2.30}) into (\ref{2.29}) gives
		\begin{align*}
			\begin{split}
				& \beta_{n,p,q } \int_\Omega f^{\alpha_0+\frac{p}{2}+t}\eta^2
				+
				\frac{4a_2t}{(2\alpha_0+2t+p-2)^2}\int_\Omega   \left|\nabla \left(f^{\frac{\alpha_0+t-1}{2}+\frac{p}{4} }\eta\right)\right|^2 \\
				\leq  &
				2(n-1)\kappa  \int_\Omega f^{\alpha_0+t+\frac{p}{2}-1}\eta^2
				+
				\frac{4(p+1)^2 }{a_2t} \int_\Omega f^{\alpha_0+\frac{p}{2}+t-1}|\nabla\eta|^2\\
				&+
				\frac{8a_2t}{(2\alpha_0+2t+p-2)^2}\int_\Omega f^{\alpha_0+t+\frac{p}{2}-1}|\nabla\eta|^2 .
			\end{split}
		\end{align*}
		We choose then $a_3$ and $a_4$ depending on $n$, $p$ and $q$ such that
		\begin{align*}
			\frac{a_3}{t}\leq  \frac{4a_2t}{(2\alpha_0+2t+p-2)^2}\quad\text{and}\quad
			\frac{8a_2t}{(2\alpha_0+2t+p-2)^2}+\frac{4(p+1)^2}{a_2t} \leq\frac{a_4}{t}.
		\end{align*}
		Hence
		\begin{align}
			\label{equa2.33}
			\begin{split}
				& \beta_{n,p,q } \int_\Omega f^{\alpha_0+\frac{p}{2}+t}\eta^2
				+
				\frac{a_3}{t}\int_\Omega   \left|\nabla \left(f^{\frac{\alpha_0+t-1}{2}+\frac{p}{4} }\eta\right)\right|^2 \\
				\leq  &
				2(n-1)\kappa  \int_\Omega f^{\alpha_0+t+\frac{p}{2}-1}\eta^2
				+
				\frac{a_4 }{t} \int_\Omega f^{\alpha_0+\frac{p}{2}+t-1}\left|\nabla\eta\right|^2 .
			\end{split}
		\end{align}
		Moreover, Saloff-Coste's Sobolev inequality tells us
		$$
		e^{-C_n(1+\sqrt{\kappa}R)}V^{\frac{2}{n}}R^{-2}\left\|f^{\frac{\alpha_0+t-1}{2}+\frac{p}{4} }\eta\right\|_{L^{\frac{2n}{n-2}}(\Omega)}^2\leq \int_{\Omega}\left|\nabla \left(f^{\frac{\alpha_0+t-1}{2}+\frac{p}{4} }\eta\right)\right|^2+R^{-2}\int_\Omega f^{ \alpha_0+t +\frac{p}{2} -1 }\eta ^2.
		$$
		Substituting the above into \eqref{equa2.33} yields
		\begin{align}
			\label{3.32}
			\begin{split}
				& \beta_{n,p,q } \int_\Omega f^{\alpha_0+\frac{p}{2}+t}\eta^2
				+
				\frac{a_3}{t }e^{-C_n(1+\sqrt{\kappa}R)}V^{\frac{2}{n}}R^{-2}\left\|f^{\frac{\alpha_0+t-1}{2}+\frac{p}{4} }\eta\right\|_{L^{\frac{2n}{n-2}}}^2\\
				\leq  &
				2(n-1)\kappa  \int_\Omega f^{\alpha_0+t+\frac{p}{2}-1}\eta^2+\frac{a_4}{ t }\int_\Omega f^{\alpha_0+t+\frac{p}{2}-1}|\nabla\eta|^2
				+\frac{a_3}{ t }\int_\Omega R^{-2}f^{ \alpha_0 +\frac{p}{2}+t-1 }\eta ^2.
			\end{split}
		\end{align}

		Now we set $t_0 = c_1({n,p,q })(1+\sqrt{\kappa} R)$ where $$c_{1}(n,p,q )=\max\left\{C_n, \frac{ a_1^2}{a_2\beta_{n,p,q }}, 2\right\},$$
		and choose $t$ such that $t\geq t_0$. Since
		\begin{align*}
			2(n-1)\kappa  R^2\leq\frac{ 2(n-1)}{c_1^2(n,p,q ) }t_0^2\quad \text{ and}\quad \frac{a_3}{t}\leq \frac{a_3}{c_1(n,p,q ) },
		\end{align*}
		there exists $a_5 = a_5(n,p,q)>0$ such that
		\begin{align}\label{a5}
			2(n-1)\kappa  R^2+\frac{a_3}{t}\leq a_5t_0^2 = a_5 c_1^2(n,p,q ) \left(1+\sqrt{\kappa} R\right)^2.
		\end{align}
		It follows from (\ref{3.32}) and (\ref{a5}) that
		\begin{align}\label{2.34}
			\begin{split}
				& \beta_{n,p,q } \int_\Omega f^{\alpha_0+\frac{p}{2}+t}\eta^2
				+
				\frac{a_3}{ t }e^{-t_0}V^{\frac{2}{n}}R^{-2}\left\|f^{\frac{\alpha_0+t-1}{2}+\frac{p}{4} }\eta\right\|_{L^{\frac{2n}{n-2}}}^2\\
				\leq  &
				a_5t_0^2R^{-2} \int_\Omega f^{\alpha_0+\frac{p}{2}+t-1}\eta^2+\frac{a_4}{t }\int_\Omega f^{\alpha_0+\frac{p}{2}+t-1}|\nabla\eta|^2.
			\end{split}
		\end{align}
		This is the required inequality and we finish the proof of this lemma.
	\end{proof}	
	
	\subsection{$L^{\beta}$ bound of gradient in a ball with radius $3R/4$}\label{sec3.3}
	Next, we turn to giving the following $L^{\beta}$ upper bound of the gradient of positive solutions to equation \eqref{equ0}.
	\begin{lem}\label{lpbound}
		Let $(M,g)$ be a complete manifold with $Ric\geq-(n-1)\kappa$ and $$\beta = \left(\alpha_0+t_0+\frac{p}{2}-1\right)\frac{n}{n-2}.$$
		Assume $u$ is a positive solution to equation \eqref{equ21} on the geodesic ball $B(o,R)\subset M$ and $f=|\nabla u|^2$. Then there exists $a_8 = a_8(n,p, q)>0$ such that
		\begin{align}\label{lpbpund}
			\|f \|_{L^{\beta}(B_{3R/4}(o))}\leq a_8V^{\frac{1}{\beta}} \frac{t_0^2}{ R^2},
		\end{align}
		where $V$ is the volume of geodesic ball $B_R(o)$.
	\end{lem}
	
	\begin{proof}
		If
		$$
		f\geq  \frac{2 a_{5}t_0^2}{\beta_{n,p,q }R^2},
		$$
		then we can obtain from (\ref{2.34})
		$$
		a_5t_0^2R^{-2} \int_\Omega f^{\alpha_0+\frac{p}{2}+t-1}\eta^2
		\leq
		\frac{\beta_{n,p,q }}{2}\int_\Omega f^{\alpha_0+\frac{p}{2}+t}\eta^2.
		$$
		We denote $\Omega_1 = \{f\geq  \frac{2a_{5}t_0^2}{\beta_{n,p,q }R^2} \}$ and $\Omega_2=\Omega\backslash\Omega_1$. Then, it is not difficult to see
		\begin{align}\label{2.36}
			\begin{split}
				a_5t_0^2R^{-2} \int_\Omega f^{\alpha_0+\frac{p}{2}+t-1}\eta^2
				=&a_5t_0^2R^{-2} \int_{\Omega_1} f^{\alpha_0+\frac{p}{2}+t-1}\eta^2
				+a_5t_0^2R^{-2} \int_{\Omega_2} f^{\alpha_0+\frac{p}{2}+t-1}\eta^2
				\\
				\leq&
				\frac{\beta_{n,p,q }}{2}\int_\Omega f^{\alpha_0+\frac{p}{2}+t}\eta^2+
				\frac{2a_5t_0^2}{R^2} \left(\frac{2a_{5}t_0^2}{\beta_{n,p,q }R^2}\right)^{\alpha_0+\frac{p}{2} +t-1 }V,
			\end{split}
		\end{align}
		where $V$ is the volume of $B_R(o)$. By choosing $t=t_0$ we can obtain from (\ref{2.34}) and (\ref{2.36})
		\begin{align}\label{2.37}
			\begin{split}
				&\frac{\beta_{n,p,q }}{2}\int_\Omega f^{\alpha_0+\frac{p}{2}+t_0}\eta^2+ \frac{a_3}{ t_0 }e^{-t_0}V^{\frac{2}{n}}R^{-2}\left\|f^{\frac{\alpha_0+t_0-1}{2}+\frac{p}{4} }\eta\right\|_{L^{\frac{2n}{n-2}}}^2\\
				\leq &\frac{2a_5t_0^2}{R^2} \left(\frac{2a_{5}t_0^2}{\beta_{n,p,q }R^2}\right)^{\alpha_0+\frac{p}{2} +t_0-1 }V
				+\frac{a_4}{ t_0 }\int_\Omega f^{\alpha_0+\frac{p}{2}+t_0-1}|\nabla\eta|^2.
			\end{split}
		\end{align}
		
		Now, we choose $\eta_1\in C^{\infty}_0(B_R(o))$ satisfying
		$$
		\begin{cases}
			0\leq\eta_1\leq 1,\quad \eta_1\equiv 1\text{ in }  B_{\frac{3R}{4}}(o);\\
			|\nabla\eta_1|\leq\frac{C(n)}{R},
		\end{cases}
		$$
		and let $\eta = \eta_1^{ \alpha_0 + \frac{p}{2}+t_0}$ in \eqref{2.37}. Then, we take a direct calculation to derive
		\begin{align*}
			a_4R^2 |\nabla\eta|^2\leq a_4 C^2(n)\left( t_0+ \frac{p}{2}+\alpha\right )^2\eta ^{\frac{2\alpha_0+2t_0+p-2}{\alpha_0+p/2+t_0}}
			\leq a_{6}t^2_0\eta^{\frac{2\alpha_0+p+2t_0-2}{\alpha_0+p/2+t_0}}.
		\end{align*}
		By H\"older inequality and Young inequality, we have
		\begin{align}\label{2.39}
			\begin{split}
				\frac{a_4}{t_0}\int_{\Omega}f^{\frac{p}{2}+\alpha+t_0-1}|\nabla\eta|^2
				\leq &\frac{a_6t_0}{R^2} \int_{\Omega}f^{\frac{p}{2}+\alpha_0+t_0-1}
				\eta^{\frac{2\alpha_0+p+2t_0-2}{\alpha+p/2+t_0}}\\
				\leq &\frac{a_6t_0}{R^2}  \left(\int_{\Omega}f^{\alpha_0+t_0+ \frac{p}{2} }\eta^2\right)^{\frac{\alpha_0+p/2+t_0-1}{\alpha+p/2+t_0}}V^{\frac{ 1}{\alpha_0+t_0+ p/2 }}\\
				\leq &\frac{\beta_{n,p,q }}{2}\left[\int_{\Omega}f^{ \alpha_0+t_0+\frac{p}{2} }\eta^2 + \left(\frac{2a_{6}t_0 }{\beta_{n,p,q }R^2}\right)^{ \alpha_0+t_0+p/2 }V\right].
			\end{split}
		\end{align}
		By combining (\ref{2.37}) and (\ref{2.39}) we obtain
		\begin{align}\label{2.40}
			\begin{split}
				&\left(\int_{\Omega}f^{\frac{n(p/2+\alpha_0+t_0-1)}{n-2}}\eta^{\frac{2n}{n-2}}\right)^{\frac{n-2}{n}}\\
				\leq &\frac{t_0}{a_3} e^{t_0}V^{1-\frac{2}{n}}R^2\left[\frac{2a_5t_0^2}{R^2} \left(\frac{2a_{5}t_0^2}{\beta_{n,p,q }R^2}\right)^{t_0+\frac{p}{2}+\alpha_0-1} + \frac{a_{6}t_0}{R^2} \left(\frac{2a_{6}t_0 }{\beta_{n,p,q }R^2}\right)^{\alpha_0+t_0+\frac{p}{2} -1 }\right]\\
				\leq &a_7^{\alpha_0+t_0+\frac{p}{2} -1}e^{t_0}V^{1-\frac{2}{n}}t_0^3\left( \frac{t_0^2}{ R^2}\right)^{\alpha_0+t_0+\frac{p}{2} -1},
			\end{split}
		\end{align}
		where $a_7$ depending only on $n$, $p$ and $q$ satisfies
		$$
		a_7^{\alpha_0+t_0+\frac{p}{2} -1} \geq \frac{2a_5}{a_3}\left(\frac{2a_5}{\beta_{n,p,q }}\right)^{\alpha_0+t_0+\frac{p}{2}-1}
		+\frac{a_6}{a_3t_0}\left(\frac{2a_6}{\beta_{n,p,q }t_0}\right)^{\alpha_0+t_0+\frac{p}{2}-1}.
		$$
		Thus
		\begin{align*}
			\left\|f\eta^{\frac{2}{\alpha_0+t_0+p/2-1}}\right\|_{L^{\beta}(\Omega)}\leq a_7e^{\frac{t_0}{\alpha_0+t_0+\frac{p}{2} -1}}V^{1/\beta}
			t_0^{\frac{3}{\alpha_0+t_0+p/2-1}}\frac{t_0^2}{ R^2}\leq a_8V^{\frac{1}{\beta}} \frac{t_0^2}{ R^2},
		\end{align*}
		where $a_8$ depending only on $n,p,q$ satisfies
		$$
		a_8 = a_7 \sup_{t_0\geq 1}t_0^{\frac{3}{\alpha_0+t_0+p/2-1}}e^{\frac{t_0}{\alpha_0+t_0+\frac{p}{2} -1}}.
		$$
		Since $\eta\equiv1$ in $B_{3R/4}$, we obtain that
		$$
		\|f \|_{L^{\beta}(B_{3R/4}(o))}\leq a_8V^{\frac{1}{\beta}} \frac{t_0^2}{ R^2}.
		$$
		Thus, we complete the proof of this lemma.
	\end{proof}

	\subsection{Moser iteration}\label{sec3.4}
	\begin{lem}\label{bound}
		Let $(M,g)$ be a complete manifold with $Ric\geq-(n-1)\kappa$. Assume $u$ is a solution to equation \eqref{equ21} on the geodesic ball $B(o,R)\subset M$ and $f=|\nabla u|^2$. Then there exists $a_{11} = a_{11}(n,p,q)>0$ such that
		\begin{align*}
			\|f\|_{L^{\infty}(B_{R/2}(o))}\leq & a_{11}\frac{(1+\sqrt{\kappa}R)^2}{R^2}.
		\end{align*}
	\end{lem}
	
	\begin{proof}
		We discard the first term of the left hand side of (\ref{2.34}) to obtain
		\begin{align}\label{2.42}
			\frac{a_3}{ t }e^{-t_0}V^{\frac{2}{n}}R^{-2}\left\|f^{\frac{\alpha_0+t-1}{2}+\frac{p}{4} }\eta\right\|_{L^{\frac{2n}{n-2}}}^2
			\leq a_5t_0^2R^{-2} \int_\Omega f^{\alpha_0+\frac{p}{2}+t-1}\eta^2+\frac{a_4}{ t }\int_\Omega f^{\alpha_0+\frac{p}{2}+t-1}|\nabla\eta|^2.
		\end{align}
		Let $r_k = \frac{R}{2}+\frac{R}{4^k}$ and $\Omega_k = B_{r_k}(o)$. It is easy to see that there exist cut-off functions $\eta_k\in C^{\infty}(\Omega_k )$ satisfying
		\begin{align*}
			\begin{cases}
				0\leq \eta_k\leq1, \quad 	|\nabla\eta_k|\leq \frac{C4^k}{R};\\
				\eta_k\equiv 1\text{ in }B_{r_{k+1}}(o),
			\end{cases}
		\end{align*}
		where $k=1,\, 2,\, 3,\, \cdots$. Substituting $\eta_k$ into (\ref{2.42}) instead of $\eta$, we arrive at
		\begin{align}\label{add}\nonumber
			a_3e^{-t_0}V^{\frac{2}{n}} \left\|f^{\frac{\alpha_0+t-1}{2}+\frac{p}{4} }\eta_k\right\|_{L^{\frac{2n}{n-2}}(\Omega_k)}^2
			\leq  & a_5t_0^2t\int_{\Omega_k} f^{\alpha_0+\frac{p}{2}+t-1}\eta_k^2+ a_4R^2\int_{\Omega_k} f^{\alpha_0+\frac{p}{2}+t-1}\left|\nabla\eta_k\right|^2\\
			\leq & \left(a_5t_0^2t + C^216^k\right)\int_{\Omega_k} f^{\alpha_0+\frac{p}{2}+t-1}.
		\end{align}
		By picking $\beta_1=\beta$ and $\beta_{k+1}=\frac{n\beta_k}{n-2}$, and letting $t=t_k$ such that $$t_k+\frac{p}{2}+\alpha-1=\beta_k,$$
		we can deduce from \eqref{add} that
		\begin{align*}
			a_3 \left(\int_{\Omega_k}f^{\beta_{k+1}}\eta_k^\frac{2n}{n-2}\right)^{\frac{n-2}{n}}
			\leq & e^{ t_0}V^{-\frac{2}{n}}\left(a_5t_0^2\left(t_0+\frac{p}{2}+\alpha_0-1\right)\left(\frac{n}{n-2}\right)^k + C^216^k\right)\int_{\Omega_k} f^{\beta_k},
		\end{align*}
		where $k=1,\,2,\, 3\,\cdots$.
		
		On the other hand, we can choose $a_9=a_9(n,p,q)$ which satisfies
		$$
		a_9t_0^3\geq \max\left\{a_5t_0^2\left(\alpha_0+t_0+\frac{p}{2}-1\right),\, C^2 \right\},
		$$
		since $\frac{n}{n-2}<16$. Then we have
		\begin{align}\label{eq:2.45}
			a_3 \left(\int_{\Omega_k}f^{\beta_{k+1}}\eta_k^\frac{2n}{n-2}\right)^{\frac{n-2}{n}}
			\leq &2a_9t_0^3 e^{ t_0}V^{-\frac{2}{n}} 16^k \int_{\Omega_k} f^{\beta_k}.
		\end{align}
		Taking $\frac{1}{\beta_k}$ power of the both sides of (\ref{eq:2.45}), we obtain
		\begin{align*}
			\|f\|_{L^{\beta_{k+1}}(\Omega_{k+1})}
			\leq &\left(2a_9t_0^3 e^{ t_0}V^{-\frac{2}{n}}\right)^{\frac{1}{\beta_k}} 16 ^{\frac{k}{\beta_k}}\|f\|_{L^{\beta_k}(\Omega_k)}.
		\end{align*}
		Noting
		$$\sum_{k=1}^{\infty}\frac{1}{\beta_k} = \frac{\frac{1}{\beta_1}}{ 1-\frac{n-2}{n}} =\frac{n}{2\beta_1}\quad\mbox{and}\quad \sum_{k=1}^{\infty}\frac{k}{\beta_k} <\infty,$$
		we can derive
		\begin{align*}
			\|f\|_{L^{\infty}(B_{R/2}(o))}\leq &a_{10}  V^{-\frac{1}{\beta}} \|f\|_{L^{\beta}(B_{3R/4}(o))},
		\end{align*}
		where $a_{10}$ depending only on $n,p,q$ satisfies
		$$
		a_{10} \geq \left(2 a_9t_0^3 e^{ t_0} \right)^{\frac{n}{2\beta_1 }} 16 ^{\sum_{k=1}^{\infty}\frac{k}{\beta_k}}.
		$$
		In view of (\ref{lpbpund}), we obtain
		\begin{align*}
			\|f\|_{L^{\infty}(B_{R/2}(o))}\leq & a_{11}\frac{(1+\sqrt{\kappa}R)^2}{R^2},
		\end{align*}
		where $a_{11} = a_{10}a_8c_1(n,p,q )$.
	\end{proof}
	
	Recalling $u = -(p-1)\log v$ where $v$ is a positive solution to \eqref{equ0}, actually we can conclude from \lemref{bound} the following
	\begin{thm}[=\thmref{thm1}]
		Let $(M,g)$ be an complete Riemannian manifold with $\ric_g\geq-(n-1)\kappa g$, where $\kappa$ is a non-positive constant.
		Assume $a$, $q$ and $p\ (p>1)$ satisfy one of the following two conditions,
		\begin{enumerate}
			\item $$a\left(\frac{n+1}{n-1}-\frac{q}{p-1}\right)\geq0,$$
			
			\item
			$$p-1<q<\frac{n+3}{n-1}(p-1).$$
		\end{enumerate}
		Then for any positive solution of $\Delta_pv +av^q=0$ on a geodesic ball $B_R(o)\subset M$, we have
		$$
		\sup_{B_{\frac{R}{2}}(o)} \frac{|\nabla v|^2}{v^2}\leq c(n,p,q)\frac{(1+\sqrt\kappa R)^2}{R^2}.
		$$
	\end{thm}

	\begin{proof}[\bf Proof of  \corref{cor2}.]
		When $a>0$, the union of the range of $q$ in condition (\ref{cond1}) and the range of $q$ in (\ref{cond2}) is
		$$
		q<\frac{n+3}{n-1}(p-1) .
		$$
		When $a<0$, the union of the ranges of $q$ in condition (\ref{cond1}) and (\ref{cond2}) is
		$$
		q>p-1.
		$$
		Hence, \corref{cor2} can be directly deduced from \thmref{thm1}.
	\end{proof}
	
	\begin{proof}[\bf Proof of \thmref{thm2}.]
		Choosing $\kappa=0$ in \thmref{thm1} implies
		\begin{align}
			\label{2.48}
			\sup_{B_{R/2}(o)} \frac{|\nabla v|}{v}
			\leq &
			c(n,p,q)\frac{1}{R}.
		\end{align}
		By letting $R\to \infty$ in \eqref{2.48}, we obtain
		$$
		\nabla v=0.
		$$
		Hence $v$ is a constant and $ \Delta_pv=0$. This contradicts to equation \eqref{equ0} since $v$ is positive. We complete the proof.
	\end{proof}
	
	Next, we provide the proof of \thmref{thm1.6}.
	\begin{proof}[\bf Proof of \thmref{thm1.6}.]
		By \thmref{thm1}, for any $p\in M$ we have
		\begin{align*}
			|\nabla u(p)|\leq \sup_{B_{\frac{R}{2}}(p)}|\nabla u|
			\leq & c(n,p,q)\frac{ 1+\sqrt{\kappa}R }{R }.
		\end{align*}
		Letting $R\to\infty$, we obtain that
		$$
		|\nabla u(p)|\leq c(n,p,q)\sqrt{\kappa}, \quad \forall p\in M.
		$$
		Fix $x_0\in M$, for any $x\in M$, choose a minimizing geodesic $\gamma(t)$ connecting $x_0$ and $x$:
		$$
		\gamma:[0,d]\to M,\quad\gamma(0)=x_0, \quad \gamma(d)=x,
		$$
		where $d=dist(x, x_0)$ is the distance of $x_0$ and $x$. So we have
		\begin{align}\label{3.49}
			u(x)-u(x_0)=\int_0^d\frac{d}{dt}u\circ\gamma(t)dt.
		\end{align}
		Since
		\begin{align}\label{3.50}
			\left|\frac{d}{dt}u\circ\gamma(t)\right|\leq |\nabla u||\gamma'(t)| = c(n,p,q)\sqrt{\kappa},
		\end{align}
		it follows from (\ref{3.49}) and (\ref{3.50}) that
		\begin{align*}
			u(x_0)-c(n,p,q)\sqrt{\kappa}d\leq  u(x)\leq u(x_0)+c(n,p,q)\sqrt{\kappa}d.
		\end{align*}
		As $u = -(p-1)\ln v$, we can derive the required inequality. Thus we finish the proof.
	\end{proof}

	{\bf Acknowledgments:} Y. Wang is supported partially by National Natural Science Foundation of China (Grant No.11971400); G. Wei is supported by National Natural Science Foundation of China (Grants No.12101619 and 12141106).
	\bigskip
	
	\bibliographystyle{acm}

\begin{thebibliography}{10}
		
		\bibitem{MR1004713}
		{\sc Bidaut-V\'{e}ron, M.-F.}
		\newblock Local and global behavior of solutions of quasilinear equations of
		{E}mden-{F}owler type.
		\newblock {\em Arch. Rational Mech. Anal. 107}, 4 (1989), 293--324.
		
		\bibitem{MR1134481}
		{\sc Bidaut-V\'{e}ron, M.-F., and V\'{e}ron, L.}
		\newblock Nonlinear elliptic equations on compact {R}iemannian manifolds and
		asymptotics of {E}mden equations.
		\newblock {\em Invent. Math. 106}, 3 (1991), 489--539.
		
		\bibitem{MR982351}
		{\sc Caffarelli, L.~A., Gidas, B., and Spruck, J.}
		\newblock Asymptotic symmetry and local behavior of semilinear elliptic
		equations with critical {S}obolev growth.
		\newblock {\em Comm. Pure Appl. Math. 42}, 3 (1989), 271--297.
		
		\bibitem{MR1121147}
		{\sc Chen, W.~X., and Li, C.}
		\newblock Classification of solutions of some nonlinear elliptic equations.
		\newblock {\em Duke Math. J. 63}, 3 (1991), 615--622.
		
		\bibitem{MR385749}
		{\sc Cheng, S.~Y., and Yau, S.~T.}
		\newblock Differential equations on {R}iemannian manifolds and their geometric
		applications.
		\newblock {\em Comm. Pure Appl. Math. 28}, 3 (1975), 333--354.
		
		\bibitem{MR4182319}
		{\sc Constantin, A., Crowdy, D.~G., Krishnamurthy, V.~S., and Wheeler, M.~H.}
		\newblock Stuart-type polar vortices on a rotating sphere.
		\newblock {\em Discrete Contin. Dyn. Syst. 41}, 1 (2021), 201--215.
		
		\bibitem{MR4444081}
		{\sc Constantin, A., and Germain, P.}
		\newblock Stratospheric planetary flows from the perspective of the {E}uler
		equation on a rotating sphere.
		\newblock {\em Arch. Ration. Mech. Anal. 245}, 1 (2022), 587--644.
		
		\bibitem{MR0709038}
		{\sc DiBenedetto, E.}
		\newblock {$C\sp{1+\alpha }$} local regularity of weak solutions of degenerate
		elliptic equations.
		\newblock {\em Nonlinear Anal. 7}, 8 (1983), 827--850.
		
		\bibitem{MR615628}
		{\sc Gidas, B., and Spruck, J.}
		\newblock Global and local behavior of positive solutions of nonlinear elliptic
		equations.
		\newblock {\em Comm. Pure Appl. Math. 34}, 4 (1981), 525--598.
		
		\bibitem{MR3225632}
		{\sc Grigor'yan, A., and Sun, Y.}
		\newblock On nonnegative solutions of the inequality {$\Delta u+u^\sigma\leq0$}
		on {R}iemannian manifolds.
		\newblock {\em Comm. Pure Appl. Math. 67}, 8 (2014), 1336--1352.
		
		\bibitem{huang2023gradient}
		{\sc Huang, G., Guo, Q., and Guo, L.}
		\newblock Gradient estimates for positive weak solution to
		$\delta_pu+au^{\sigma}=0$ on riemannian manifolds.
		\newblock {\em arXiv preprint arXiv:2304.04357\/} (2023).
		
		\bibitem{MR2518892}
		{\sc Kotschwar, B., and Ni, L.}
		\newblock Local gradient estimates of {$p$}-harmonic functions, {$1/H$}-flow,
		and an entropy formula.
		\newblock {\em Ann. Sci. \'{E}c. Norm. Sup\'{e}r. (4) 42}, 1 (2009), 1--36.
		
		\bibitem{MR834612}
		{\sc Li, P., and Yau, S.-T.}
		\newblock On the parabolic kernel of the {S}chr\"{o}dinger operator.
		\newblock {\em Acta Math. 156}, 3-4 (1986), 153--201.
		
		\bibitem{MR829846}
		{\sc Ni, W.-M., and Serrin, J.}
		\newblock Nonexistence theorems for singular solutions of quasilinear partial
		differential equations.
		\newblock {\em Comm. Pure Appl. Math. 39}, 3 (1986), 379--399.
		
		\bibitem{peng2020gradient}
		{\sc Peng, B., Wang, Y., and Wei, G.}
		\newblock Gradient estimates for {$\Delta u+ au^{p+ 1}= 0$} and liouville
		theorems.
		\newblock {\em arXiv preprint arXiv:2009.14566\/} (2020).
		
		\bibitem{saloff1992uniformly}
		{\sc Saloff-Coste, L.}
		\newblock Uniformly elliptic operators on riemannian manifolds.
		\newblock {\em Journal of Differential Geometry 36}, 2 (1992), 417--450.
		
		\bibitem{MR788292}
		{\sc Schoen, R.}
		\newblock Conformal deformation of a {R}iemannian metric to constant scalar
		curvature.
		\newblock {\em J. Differential Geom. 20}, 2 (1984), 479--495.
		
		\bibitem{MR929283}
		{\sc Schoen, R.}
		\newblock The existence of weak solutions with prescribed singular behavior for
		a conformally invariant scalar equation.
		\newblock {\em Comm. Pure Appl. Math. 41}, 3 (1988), 317--392.
		
		\bibitem{MR931204}
		{\sc Schoen, R., and Yau, S.~T.}
		\newblock Conformally flat manifolds, {K}leinian groups and scalar curvature.
		\newblock {\em Invent. Math. 92}, 1 (1988), 47--71.
		
		\bibitem{MR1946918}
		{\sc Serrin, J., and Zou, H.}
		\newblock Cauchy-{L}iouville and universal boundedness theorems for quasilinear
		elliptic equations and inequalities.
		\newblock {\em Acta Math. 189}, 1 (2002), 79--142.
		
		\bibitem{MR3336621}
		{\sc Sun, Y.}
		\newblock On nonexistence of positive solutions of quasi-linear inequality on
		{R}iemannian manifolds.
		\newblock {\em Proc. Amer. Math. Soc. 143}, 7 (2015), 2969--2984.
		
		\bibitem{MR3275651}
		{\sc Sung, C.-J.~A., and Wang, J.}
		\newblock Sharp gradient estimate and spectral rigidity for {$p$}-{L}aplacian.
		\newblock {\em Math. Res. Lett. 21}, 4 (2014), 885--904.
		
		\bibitem{MR0727034}
		{\sc Tolksdorf, P.}
		\newblock Regularity for a more general class of quasilinear elliptic
		equations.
		\newblock {\em J. Differential Equations 51}, 1 (1984), 126--150.
		
		\bibitem{MR0474389}
		{\sc Uhlenbeck, K.}
		\newblock Regularity for a class of non-linear elliptic systems.
		\newblock {\em Acta Math. 138}, 3-4 (1977), 219--240.
		
		\bibitem{MR2880214}
		{\sc Wang, X., and Zhang, L.}
		\newblock Local gradient estimate for {$p$}-harmonic functions on {R}iemannian
		manifolds.
		\newblock {\em Comm. Anal. Geom. 19}, 4 (2011), 759--771.
		
		\bibitem{MR4559367}
		{\sc Wang, Y., and Wei, G.}
		\newblock On the nonexistence of positive solution to {$\Delta u + au^{p+1} =
			0$} on {R}iemannian manifolds.
		\newblock {\em J. Differential Equations 362\/} (2023), 74--87.
		
		\bibitem{MR431040}
		{\sc Yau, S.~T.}
		\newblock Harmonic functions on complete {R}iemannian manifolds.
		\newblock {\em Comm. Pure Appl. Math. 28\/} (1975), 201--228.
		
		\bibitem{MR3866881}
		{\sc Zhao, L., and Yang, D.}
		\newblock Gradient estimates for the {$p$}-{L}aplacian {L}ichnerowicz equation
		on smooth metric measure spaces.
		\newblock {\em Proc. Amer. Math. Soc. 146}, 12 (2018), 5451--5461.
		
	\end{thebibliography}

\end{document}